\documentclass[12pt]{article}
\usepackage{amscd}
\usepackage{amsmath}
\usepackage{amssymb}
\usepackage{verbatim}
\usepackage{amsthm}
\usepackage[truedimen,margin=25truemm]{geometry}

\numberwithin{equation}{section}
\newtheorem{thm}{Theorem}[section]
\newtheorem{prop}[thm]{Proposition}
\newtheorem{lem}[thm]{Lemma}

\newtheorem{df}[thm]{Definition}
\newtheorem{rem}[thm]{Remark}

\newcommand{\C}{\mathbb{C}}
\newcommand{\Cwp}{\mathbb{C}_\wp}
\newcommand{\End}{\mathrm{End}}
\newcommand{\F}{\mathbb{F}}
\newcommand{\G}{\mathbb{G}}
\newcommand{\Ga}{\mathbb{G}_a}

\newcommand{\cO}{\mathcal{O}}
\newcommand{\Spec}{\mathrm{Spec}\,}
\newcommand{\vp}{\varphi}

\title{An explicit form of the canonical submodule of a Drinfeld module}

\author{Satoshi KONDO\footnote{Faculty of Mathematics, National Research University Higher School of Economics, Moscow, Russia; Kavli Institute for the Physics and Mathematics of the Universe, Japan}, Yusuke SUGIYAMA\footnote{Department of Mathematics, Graduate School of Science, Osaka university}}

\begin{document}
	
\maketitle

\renewcommand{\thefootnote}{\fnsymbol{footnote}}

\begin{abstract}
We define the canonical submodule of a Drinfeld module of 
rank greater than one 
over the affine line over a finite field.
(This extends the definition of the level 1 canonical subgroup of Hattori for rank 2 with ordinary reduction.)
We give a criterion for the existence of the canonical submodule in terms of a lift of the Hasse invariant.
Then, we give an explicit form of the canonical submodule.    
The main tool is formal Drinfeld modules of Rosen.

Recall that a canonical subgroup of an elliptic curve plays an important role in the theory of $p$-adic modular forms.
An explicit form in this case is given by Coleman, and our result is 
its function field analogue.
\end{abstract}

\section{Introduction}
In this paper, we define the canonical submodule of a Drinfeld module 
(Definitions~\ref{canonical sub},~\ref{canonical subsch}), 
give a criterion (Theorem~\ref{condition 1})
for the existence 
in terms of a lift of the Hasse invariant 
(Remark~\ref{Hasse}).   
Then we give an explicit form 
(the affine coordinate ring) 
of the canonical submodule 
(Theorem~\ref{main thm}).

Canonical subgroups of elliptic curves play an important role in the theory of $p$-adic modular forms \cite{K}.  
An explicit form is given by Coleman \cite{C}.   Let us recall his result.
Let $p>3$ be a prime. We denote by $R$ a subring of the ring of integers of $\C_p$ (the completion of an algebraic 
closure of $\mathbb{Q}_p$). Note that for a pair $(E,\omega)$ of an elliptic curve over $R$ and its invariant differential, the Hasse invariant of the pair $(E,\omega)\mod p$ is lifted to the associated Eisenstein series $E_{p-1}(E,\omega)$ of weight $p-1$ since $p>3$.
Coleman showed that the canonical subgroup of an elliptic curve $E$ is isomorphic to the subgroup scheme of 
the additive group scheme $\G_{a, R}$ given by
$$
\Spec R[x]/(x^p+(p/E_{p-1}(E,\omega))x),
$$ 
(cf \cite{C} p.4 Theorem 2.1.).

Recently a function field analogue of the $p$-adic modular forms theory has been developed by Goss, Hattori and many others (see e.g. \cite{G},\cite{H1},\cite{H2}). 
In this note, we consider a function field analogue of Coleman's result above.

One difference between the case of elliptic curves and of Drinfeld modules is that the canonical subgroup is equipped 
with the structure of an $A$-module (where $A=\F_q[T])$), as any Drinfeld module is an $A$-module.    
We therefore call it a canonical ($A$-)submodule 
rather than a canonical subgroup, and, 
for the explicit description,
we also give the structure as an $A$-module explicitly.

The paper is organized as follows.   
In Section~\ref{sec:statement}, 
we define the {\itshape{canonical submodule (scheme)}} of a Drinfeld module and give the statement of our 
main result (Theorem~\ref{main thm}).
In Section~\ref{sec:canonical submodules}, we study Newton polygons and give a criterion for the existence 
of canonical submodule schemes. Section~\ref{sec:proof} is 
devoted to the proof of Theorem~\ref{main thm}.

{\bf{Acknowledgment}}
We thank Shin Hattori for suggesting 
this topic and for numerous valuable comments.
The first author is supported by WPI initiative, MEXT, Japan. 
The second author expresses 
his gratitude to his supervisor 
Professor Seidai Yasuda for his advice and encouragement.

\section{Statement of the main result}
\label{sec:statement}
The main aim is to give the statement of the main theorem
(Theorem~\ref{main thm}).  
\subsection{Setup}
Let $A=\F_q[T]$  be the polynomial ring over a finite field of $q$ elements. 
Fix a prime ideal $\wp$ of $A$ generated by a monic irreducible polynomial 
$\varpi=\varpi(T) \in A$ of degree $d\geq 1$. 
Set $\kappa(\wp)=A/\wp$. 
We denote by $A_{\wp}$ the 
$\wp$-adic completion of $A$. 
Note that $A_{\wp}$ contains a subfield, denoted by $\F$, 
that is isomorphic to $\kappa(\wp)$. 
We let $\iota$ denote the isomorphism
$\iota: \F \subset A_\wp \to A_\wp/\wp A_\wp$
where the arrow is the canonical quotient map.
The $\wp$-adic completion of an algebraic closure 
of the fraction field of $A_{\wp}$ is denoted by $\C_{\wp}$. 
We denote by $v(\star)=v_{\wp}(\star)$ (resp.\ $\left| \star \right| $)
the valuation (resp.\ absolute value) on 
$\C_{\wp}$ normalized as $\left| \varpi \right| = q^{-v(\varpi)}=q^{-d}$. 

Throughout this article, we let $\cO$ 
be a local subring of $\C_{\wp}$ 
with the maximal ideal 
$\mathfrak{m}=\lbrace x \in \cO \mid v(x)  > 0 \rbrace$.
Moreover we assume that the ring $\cO$ contains $A_{\wp}$ and that the ring $\cO$ is complete with respect to the absolute value $\left| \star \right| $.

Let $(E, \varphi)$ be a Drinfeld module over 
$\cO$ of rank $r\geq 2$.  
This $E$ is the additive group scheme 
$\Ga=\G_{a,O}$ over $\cO$
and $\varphi: A \to \End(E)$ is a
ring homomorphism.   Here, the endomorphisms are as 
a group scheme over $\cO$.
We fix a parameter $X$ of the additive group scheme 
$E=\Ga=\G_{a,\cO}$. Thus we may identify the endomorphism ring $\End(E)$ 
with a subring (by composition) of $\cO[X]$. 
Then we write the action of $\varpi \in A$ as
$$
\varphi_{\varpi}(X)=\sum_{i=0}^{rd}H_i X^{q^i} \in \cO[X].
$$
 Note that we have $H_0=\varpi$ and $H_{rd} \in \cO^{\times}$. 

\begin{rem}[a lift of the Hasse invariant]
\label{Hasse}
The Hasse invariant of a rank two Drinfeld module over an algebraic closure of $A/\wp$
is defined by Gekeler.    We are particularly interested in the coefficient $H_d \in \cO$.
This is a lift of the Hasse invariant.
\end{rem}

\subsection{Canonical submodule schemes}

In this subsection, we define the canonical 
submodule of a Drinfeld module in a way similar to that 
in \cite{L}. 

Let $\varphi$ be a Drinfeld module over 
$\mathcal{O}$ of rank $r\geq 2$. 
We denote by $\varphi[\varpi]$ 
the $\varpi$-torsion subscheme 
of $\mathbb{G}_a$,
and by 
$\varphi[\varpi](\C_{\wp})$ 
the set of $\C_{\wp}$-valued points. 
Note that $\varphi[\varpi](\C_{\wp})$ 
is a free $\kappa(\wp)$-module 
of rank $r$ via $\varphi$. 

For a non-negative real number $\delta \in \mathbb{R}$, we set

$$
B(\delta)=
\left\lbrace 
x \in \C_{\wp} \mid \left| x \right| \leq \delta
\right\rbrace .
$$
Note that $\varphi[\varpi](\C_{\wp}) \cap B(\delta)$ is an $A$-submodule of $\varphi[\varpi](\C_{\wp})$ (and thus $\kappa(\wp)$-submodule since it is $\wp$-torsion) via $\varphi$. Moreover we set 
$$
\delta_0=
\min
\left\lbrace 
\left| x\right| \mid 0\neq x \in \varphi[\varpi](\C_{\wp})
\right\rbrace .
$$

\begin{df}[Canonical Submodule]\label{canonical sub}
If the $A$-submodule $\varphi[\varpi]_{\leq \delta_0}=
\varphi[\varpi](\C_{\wp}) \cap B(\delta_0)$ 
of
 $\varphi[\varpi](\C_{\wp})$ is a free 
$\kappa(\wp)$-module of rank $1$ via 
$\varphi$, it is called the canonical submodule of $\varphi$. 
\end{df}

\begin{df}[Canonical Submodule Scheme]
\label{canonical subsch}
Let $C \subset \varphi[\varpi]$ be an $A$-submodule scheme defined over $\cO$. We call $C$ the canonical submodule scheme of $\varphi$ if the set of $\C_{\wp}$-valued points of $C$ is the canonical submodule of $\varphi$ in the sense of Definition \ref{canonical sub}.
\end{df}
Hattori has already defined the {\itshape{canonical subgroup}}, which is similar to the canonical submodule scheme, for a Drinfend module with ordinary reduction in \cite{H1}.

The following proposition allows us to identify the canonical submodule with the canonical submodule scheme.

\begin{prop}\label{sub iff subsch}
	A Drinfeld module $\varphi$ over $\cO$ admits the canonical submodule if and only if $\varphi$ admits the canonical submodule scheme.
\end{prop}

\begin{proof}
	Suppose $\varphi$ admits the canonical submodule $\varphi[\varpi]_{\leq \delta_0}$. It suffices to prove that 
	$$
	\Phi_{0}(X)=\prod_{\alpha \in \varphi[\varpi]_{\leq \delta_0}}(X-\alpha)
	$$
	belongs to $\cO[X]$. 
	Let $K$ be the field of fractions of $\cO$ and 
	$K^{\mathrm{sep}}$ be a separable closure of $K$.  
	Note that the elements of 
	$\varphi[\varpi]_{\leq \delta_0} \subset K^{\mathrm{sep}}$ 
	are Galois conjugates of one another,
	since each $\alpha \in \varphi[\varpi]_{\leq \delta_0}$ 
	is a root of the separable polynomial 
	${\vp}_{\varpi}(X) \in \cO[X]$ and 
	the valuation is invariant 
	under the Galois action. 
	Thus we have $\Phi_{0}(X) \in \cO[X]$. The converse is clear.
\end{proof}

\begin{rem}
	The canonical submodule scheme is a lift of $q^d$-th Frobenius on 
$E \mod \mathfrak{m}=\mathbb{G}_{a, \cO/\mathfrak{m}}$, that is, one has
	$$
	\Phi_{0}(X) \equiv X^{q^d} \mod \mathfrak{m}.
	$$
This is because we have $\Phi_{0}(X) \in \cO[X]$ (Proposition \ref{sub iff subsch}) 
and $\left| \alpha \right| <1$ 
for any $\alpha \in \varphi[\varpi]_{\leq \delta_0}$.
\end{rem}

In Section 3, we give a necessary and sufficient condition for the existence of the canonical submodule 
(Theorem~\ref{condition 1}) in terms of a Newton polygon.

\subsection{Statement of the main theorem}
\label{sec:definition of phi}
Recall that $A_{\wp}$ contains a subfield $\F$ which is isomorphic to $\kappa(\wp)=A/\wp$. 
In this subsection, 
we construct certain $\F$-module schemes.

\begin{df}
For an element $c \in \cO$, we define a closed subgroup scheme of $\mathbb{G}_{a,\cO}$ as
$$
Z(c)=\Spec  \cO[X]/(X^{q^d}+cX).
$$
\end{df}
We define an $\F$-module scheme structure of $Z(c)$,
i.e., a ring homomorphism $\psi: \F \to \End(Z(c))$
as follows.   Each $\lambda \in \F$ is sent to 
the morphism of schemes induced by the 
$\cO$-algebra homomorphism 
$\cO[X]/(X^{q^d}+cX) \to \cO[X]/(X^{q^d}+cX)$ that sends $X$ 
to $\lambda X$.

We often write $(X, f)$ for an abelian group scheme $X$ and
a ring homomorphism $f: R \to \End(X)$ to
denote the group scheme with $R$-module structure.

We define an $A$-module structure $\phi$ 
on $Z(c)$ as follows.
We set 
$\phi: A \to A_\wp \to A_\wp/\wp A_\wp \to \F \to \End(Z(c))$.
Here, the first is the canonical inclusion, the second is the canonical quotient map,
the third is $\iota^{-1}$ and the last is $\psi$.

We are ready to state our main result.
\begin{thm}\label{main thm}
	Suppose a Drinfeld module $\vp$ over $\cO$ admits the canonical submodule scheme.  As an $A$-module scheme, the canonical submodule scheme is isomorphic to 
$$
(Z({\varpi/H_d}), \phi).
$$
\end{thm}
The proof will be given in Section~\ref{sec:proof}.

\section{Existence of canonical submodules}
\label{sec:canonical submodules}
The aim is to give a criterion for 
the existence of a canonical submodule 
(Theorem~\ref{condition 1}).
Let $\varphi$ be a Drinfeld module over $\mathcal{O}$ of rank $r\geq 2$. 
We consider the formal Drinfeld module 
(in the sense of Rosen \cite{R}) 
associated with $\varphi$ 
and describe the Newton polygon of
$$
\varphi_{\varpi}(X)=\sum_{i=0}^{rd}H_i X^{q^i}.
$$

\subsection{Formal Drinfeld modules}
In this subsection, we state some basic facts on formal Drinfeld modules.  See \cite{R} for more details. In \cite{R}, Rosen defined formal Drinfeld modules as an analogue of formal groups associated to elliptic curves. For a Drinfeld module $\vp$ over $\cO$, one may define a formal Drinfeld module associated to $\vp$ in the same way as \cite{R} $\S 4$. We denote the formal Drinfeld module by $\widehat{\vp}$. 

We let $\cO[[\tau]]$ denote the ring of skew power series
with $a^q \tau= \tau a$ for $a \in \cO$.   
It is regarded as 
a subring of $\cO[[X]]$ (a ring by composition) by sending $\tau$ to $X^q$.
Then we have the following.

\begin{itemize}
	\item $\widehat{\vp} : A_{\wp} \to \cO[[\tau]]$ 
is a ring homomorphism, which defines an $A_{\wp}$-module 
structure on $\mathfrak{m}$. Note that 
that $\widehat{\vp}_a(X)$ is of the form
	 $$
	 \widehat{\vp}_a(X)=\sum_{i\geq 0}a_iX^{q^i} \in \cO[[X]]
	 $$
	 for $a \in A_{\wp}$. Moreover we have $a_0=a$.

	\item $\widehat{\vp}$ induces the $A_{\wp}$-module structure on $\mathfrak{m}$ given by
	
	$$
	a \bullet x=\widehat{\vp}_a(x)=\sum_{i\geq 0}a_ix^{q^i} \in \mathfrak{m} \ 
	$$
	for $a\in A_{\wp}, x\in \mathfrak{m}$
	
	\item $\widehat{\vp}$ coincides with $\vp$ on $A$. In particular, we have  $\widehat{\vp}_{\varpi}(X)=\vp_{\varpi}(X) \in \cO[X].$\ 
\end{itemize}

\subsection{Change of variables}

In this subsection, we make a change of variables of $\widehat{\mathbb{G}}_a$ by using a power series $g(X) \in \cO[[X]]$ in order to give a simpler description of the $\varpi$-torsion points of $\vp$.  

Recall that $\F \subset A_{\wp}$.   
We fix a generator $\sigma $ of the cyclic group 
$\F^{\times}$.
\begin{df}
	We define a power series $g(X) \in \cO[[X]]$ as follows:
	$$
	g(X)=\sum_{k=1}^{q^d-1}\sigma^{-k} \widehat{\vp}_{\sigma^k}(X).
	$$
\end{df}

\begin{rem}
	Since we have $\widehat{\vp}_{\sigma^k}(X)=\sigma^kX+O(X^2)$, we see that $g(X) \in \cO[[X]]$ is of the form $g(X)=-X+O(X^2)$. Thus there exists its composition inverse $g^{-1}(X) \in \cO[[X]]$.  
\end{rem}

\begin{df}
	For any $a \in A_{\wp}$, we set 
	$$
	(\widehat{\vp}_{a})^{g}=(\widehat{\vp}_{a})^{g}(X)=g(\widehat{\vp}_{a}(g^{-1}(X))).
	$$
\end{df}

\begin{rem}
\label{rmk:F-module structure}
	Since $g$ is additive, 
the map 
$(\widehat{\vp})^{g}: 
A_{\wp} \to \cO[[X]]$
that sends $a$ to $(\widehat{\vp}_{a})^{g}$ 
is a ring homomorphism. 
This map induces the $A_{\wp}$-module structure on $\mathfrak{m}$ given by
	
	$$
	a*x=(\widehat{\vp}_{a})^{g}(x) \in \mathfrak{m} \ 
	$$
	for $a\in A_{\wp}$ and $x\in \mathfrak{m}.$ Moreover, we have 
	$$
	\lambda*x=(\widehat{\vp}_{\lambda})^{g}(X)=\lambda X
	$$
	for any $\lambda \in \F$. 
\end{rem}

\begin{lem}\label{multi of d}
	The power series $(\widehat{\vp}_{\varpi})^{g}(X)$ is of the form
	$$
	(\widehat{\vp}_{\varpi})^{g}(X)=\sum_{i\geq 0}\widehat{H}_{id}X^{q^{id}} \ 
	$$
	for some $\widehat{H}_{id} \in \cO$ for each $i$.
\end{lem}
\begin{proof}
Recall that we have 
$\widehat{\vp}_{\varpi}(X)=
\vp_{\varpi}(X) \in \cO[X].$
The 
$\F$-linearity of $(\widehat{\vp}_{\varpi})^{g}(X)$,
i.e.,
$(\widehat{\varphi}_\varpi)^g(\lambda X)=
\lambda(\widehat{\varphi}_\varpi)^g(X)$
for $\lambda \in \F$,
follows from $(\widehat{\vp}_{\lambda})^{g}(X)=\lambda X
	$ and $\widehat{\vp}_{\lambda}(\widehat{\vp}_{\varpi}(X)) = \widehat{\vp}_{\varpi}(\widehat{\vp}_{\lambda}(X))$ for any $\lambda \in \F$. The assertion follows from the $\F$-linearity and ${}^{\sharp}\F=q^d$. 
	
\end{proof}

\subsection{Newton polygons}
We compare the Newton polygon of
$$
\varphi_{\varpi}(X)=\sum_{i=0}^{rd}H_i X^{q^i}
\ \text{and of }\ 
(\widehat{\vp}_{\varpi})^{g}(X)=\sum_{i\geq 0}\widehat{H}_{id}X^{q^{id}} .
$$
\begin{lem}\label{NP invariant}
	The negative-slope part of the Newton polygon of $\vp_{\varpi}(X)$ coincides with that of $(\widehat{\vp}_{\varpi})^{g}(X)$. In particular, the set of vertices of the Newton polygon of $\vp_{\varpi}(X)$ is contained in the set $\left\lbrace (q^{id},v(H_{id}))\right\rbrace _{0 \le i \le r}$.
\end{lem}

\begin{proof}
	Note that we have a bijection
	$$
	g: \left\lbrace \alpha \in \C_{\wp} \mid \vp_{\varpi}(\alpha)=0\right\rbrace 
	\to
	\left\lbrace \beta \in \C_{\wp} \mid 	(\widehat{\vp}_{\varpi})^{g}(\beta)=0\right\rbrace 
	,\ 
	(\alpha \mapsto g(\alpha))
	.$$
	Moreover the bijection preserves the order of zero, that is,
	$\alpha \in  \C_{\wp}$ is a zero of $\vp_{\varpi}(X)$ of order $n \geq 1$ if and only if $g(\alpha) \in  \C_{\wp} $ is a zero of $(\widehat{\vp}_{\varpi})^{g}(X)$ of order $n \geq 1$. This implies the assertion since we have $H_0=\widehat{H}_0=\varpi$.
	
\end{proof}

\subsection{The criterion}
As an easy application of Lemma~\ref{NP invariant}, we give a necessary and sufficient condition for the existence of the canonical submodule (Definition \ref{canonical sub}) in terms of the Newton polygon of $\vp_{\varpi}(X)$.
\begin{thm}\label{condition 1}
A Drinfeld module $\varphi$ over $\mathcal{O}$ admits the canonical submodule scheme 
if and only if $(q^d,v(H_d))$ is a vertex of the Newton polygon of $\varphi_{\varpi}(X)$. 
\end{thm}
\begin{proof}
Lemma~\ref{NP invariant} implies that 
the first segment of the Newton polygon of 
$\varphi_{\varpi}(X)$ 
is the segment with endpoints 
$(1,1)$ and $(q^d,v(H_d))$ 
if and only if 
$(q^d,v(H_d))$ is a vertex 
of the Newton polygon of $\varphi_{\varpi}(X)$. 
Hence the statement for canonical submodule follows from
the property of Newton polygons.
The assertion for canonical submodule scheme forllows 
from Proposition~\ref{sub iff subsch}. 
\end{proof}
We use the following lemma in the last section.
\begin{lem}\label{1-unit}
	Let $(q^{id},v(H_{id}))$ be a vertex of a segment of the Newton polygon of $\vp_{\varpi}(X)$. Suppose the slope of the segment is negative. Then the fraction ${H_{id}}/{\widehat{H}_{id}}$ is a $1$-unit, that is, we have
	$$
	\frac{H_{id}}{\widehat{H}_{id}}-1 \in \mathfrak{m}.	
	$$	
\end{lem}

\begin{proof}
Lemmas \ref{multi of d} and \ref{NP invariant} imply that 
	$$
	v(H_j) > v(H_{id}),\  v(\widehat{H}_{ld}) > v(\widehat{H}_{id})
	$$
	for any $0 \leq j < q^{id}, 0 \leq l < i $. 
	Recall that we have $\widehat{\vp}_{\varpi}(X)=\vp_{\varpi}(X) \in \cO[X].$\ The assertion follows by comparing the coefficients of $X^{q^{id}}$ of both sides of the equation
	$$
	({\vp}_{\varpi})^{g}(g(X))=g({\vp}_{\varpi}(X)).
	$$
	
\end{proof}

\section{Proof of Theorem \ref{main thm} }
\label{sec:proof}
In this section, we give a proof of Theorem \ref{main thm}.  
\subsection{}
We use the notation from Section~\ref{sec:canonical submodules}.
Let $(E, \varphi)$ be a Drinfeld module over $\mathcal{O}$ of rank $r \geq2 $ and write
$$
\varphi_{\varpi}(X)=\sum_{i=0}^{rd}H_i X^{q^i}.
$$
Suppose $(q^d,v(\widehat{H}_{d}))$ is a vertex of the Newton polygon of $(\widehat{\vp}_{\varpi})^{g}=\sum_{i\geq 0}\widehat{H}_{id}X^{q^{id}} \in \cO[[X]]$. 
Since we have $H_0=\widehat{H}_0=\varpi$ and $H_{rd} \in \cO^{\times}$, the first segment of the Newton polygon of $(\widehat{\vp}_{\varpi})^{g}(X)$ is the segment whose
endpoints are $(1,1)$ and $(q^d,v(\widehat{H}_{d}))$ (Lemma \ref{multi of d}). 
We denote by $\widehat{N}$ the set of zeros (in $\Cwp$) 
of $(\widehat{\vp}_{\varpi})^{g}(X)$ corresponding to the first segment of the Newton polygon, and set $\widehat{N}_0=\left\lbrace 0\right\rbrace \cup \widehat{N}$. 
We first study the $A$-module $\widehat{N}_0$.

\begin{lem}\label{lem1}
	For any integer $i \geq 2$ and $\alpha \in \hat N$, we have
	$$
	v\left(\frac{\widehat{H}_{id}}{\widehat{H}_d} \alpha ^{q^{id}-q^d}\right)>0.
	$$
\end{lem}
\begin{proof}
	The slope of the line passing through $(q^d,v(\widehat{H}_d))$ and $(q^{id},v(\widehat{H}_{id}))$ is
	$$
	\frac{v(\widehat{H}_{id})-v(\widehat{H}_{d})}{q^{id}-q^{d}}.
	$$
	By the property of Newton polygons, this slope is larger than that of the first segment, which is $-v(\alpha)$. It follows that
	$$
	-v(\alpha)<	\frac{v(\widehat{H}_{id})-v(\widehat{H}_{d})}{q^{id}-q^{d}}
	$$
	holds for $i\geq2$. The claim follows.
\end{proof}

\begin{prop}\label{product}
	Suppose $(q^d,v(\widehat{H}_{d}))$ is a vertex of the Newton polygon of $(\widehat{\vp}_{\varpi})^{g}$. Then we have 
	$$
	X\prod_{\alpha \in \widehat{N}}(X-\alpha)=X^{q^d}+\frac{\varpi}{(1+m){H}_d}X \in \cO[X]
	$$
	for some $m \in \mathfrak{m}$.
\end{prop}

\begin{proof}
Take an element $\alpha \in \widehat{N}$. Since ${}^{\sharp}\widehat{N}_0={}^{\sharp}\F=q^d$, 
we have 
$\widehat{N}_0=\left\lbrace \lambda \alpha 
\mid \lambda \in \F \right\rbrace$. 
Using the equality
$\prod_{\lambda \in \F}(X-\lambda)=X^{q^d}-X$, 
we obtain
$$
\prod_{\lambda \in \F}(X-\lambda \alpha)
=X^{q^d}-\alpha^{q^d-1}X.
$$
Set $m=-(1+\varpi / \alpha^{q^d-1}\widehat{H}_d)$. 
Then it suffices to show that
$m \in \mathfrak{m}$. 
From 
$$
(\widehat{\vp}_{\varpi})^{g}(\alpha)=\sum_{i \geq 0}\widehat{H}_{id}\alpha^{q^{id}}=0
$$
and
$\widehat{H}_0=\varpi$,
it follows that
$$
m=\sum_{i \geq 2} \frac{\widehat{H}_{id}}{\widehat{H}_d} \alpha ^{q^{id}-q^d}.
$$
Then the claim follows from Lemmas~\ref{lem1} and~\ref{1-unit}.

\end{proof}

\subsection{}
\begin{df}
	A polynomial $F\in \cO[X]$ is called a distinguished polynomial if $F$ is monic and all its coefficients are in $\mathfrak{m}$ except for the leading coefficient. 
\end{df}

\begin{lem}\label{lemma dist}
	Let $F\in \cO[X]$ be a distinguished polynomial. The natural inclusion $\cO[X] \hookrightarrow \cO[[X]]$ induces an $\cO$-algebra isomorphism
	$$
	\cO[X]/(F) \cong \cO[[X]]/(F) .
	$$
\end{lem}

\begin{proof}
	One can check that if $f$ is in $\cO[X]$ with $f=FG$ for some $G \in \cO[[X]]$ then $G \in \cO[X]$. This implies the injectivity. Moreover one may see that for any $G \in \cO[[X]]$ there exist $Q \in \cO[[X]]$ and $r \in \cO[X]$ such that $G=QF+r$. This implies the surjectivity.
\end{proof}

\begin{rem}\label{induced str}
	For a distinguished polynomial $F\in \cO[X]$, we denote by $j: \Spec \cO[[X]]/(F) \to \Spec \cO[X]/(F)$ the isomorphism as in Lemma \ref{lemma dist}. Suppose $\Spec \cO[[X]]/(F)$ has an $A$-module structure denoted by $\phi: A \to \mathrm{End}(\Spec \cO[[X]]/(F))$. Then the isomorphism $j$ induces the $A$-module structure on $\Spec \cO[X]/(F)$ as the 
composite of the ring homomorphisms 
	$$
	A \to \mathrm{End}(\Spec \cO[[X]]/(F)) \to \mathrm{End}(\Spec \cO[X]/(F)),
	$$
	where the first arrow is $\phi$ and the second map
is the map that sends 
$f \in \mathrm{End}(\Spec \cO[[X]]/(F))$
to 
$j \circ f \circ j^{-1}$. 
\end{rem}

\subsection{}
\begin{rem}\label{str}
	As in Proposition \ref{product}, we write 
	$$
	X\prod_{\alpha \in \widehat{N}}(X-\alpha)=X^{q^d}+\frac{\varpi}{(1+m){H}_d}X \in \cO[X]
	$$
	for some $m \in \mathfrak{m}$. Recall that 
	 $$\Spec 
\cO[[X]]\left/
\left(
X^{q^d}+\frac{\varpi}{(1+m){H}_d}X
\right)\right.
$$ 
	 is  an $A$-module via the ring homomorphism
	 $\vp^{g}: A \to \cO[[X]]$
that sends $a$ to 
$(\vp_a)^{g}(X)(=(\widehat{\vp}_a)^{g}(X)).$
	 Thus in the way of Remark \ref{induced str}, one may regard
     $$\mathrm{Spec} \left.\cO[X]\middle/\left(X^{q^d}+\frac{\varpi}{(1+m){H}_d}X\right)\right.$$ 
     as an $A$-module and we denote the $A$-module structure by the same symbol $\vp^{g}$.
\end{rem}
	
	\begin{rem}\label{unit u}
	Hensel's lemma says that there exists a unit $u \in \cO^{\times}$ such that 
	$$
	u^{q^{d}-1}=1+m.
	$$
	We obtain an $\cO$-algebra isomorphism
	$$
	\cO[[X]]\left/\left(X^{q^d}+\frac{\varpi}{(1+m){H}_d}X\right)\right. \to 
\cO[[X]]\left/\left(X^{q^d}+\frac{\varpi}{{H}_d}X\right)\right.
	$$
that sends $X$ to $uX$.
This isomorphism 
induces an $A$-module structure on 
$\Spec 
\cO[[X]]/(X^{q^d}+({\varpi}/{{H}_d})X)$. 
More explicitly, one may define the 
$A$-module structure by the ring homomorpism
	$$
	u\ \vp^{g} u^{-1}: A \to \cO[[X]],
$$
that sends $a \in A$ to 
$u(\vp_a)^{g}(u^{-1}X)(=u(\widehat{\vp}_a)^{g}(u^{-1}X))$.

Note that the polynomial $X^{q^d}+(\varpi/H_d)X$ is 
distinguished.   Hence by Remark~\ref{induced str},
	$$\Spec \cO[X]\left/\left(X^{q^d}+\frac{\varpi}{{H}_d}X\right)\right.$$ 
	is equipped with the structure of an $A$-module.   
We denote the $A$-module structure by the same symbol $u\ \vp^{g} u^{-1}$.
\end{rem}

\subsection{}
We denote by $(G,\psi)$ a pair of an abelian group scheme $G$ and an $A$-module structure 
$\psi: A \to \mathrm{End}(G)$.  
Denote by $N$ the set of zeros of ${\vp}_{\varpi}(X)$ corresponding to the first segment of the Newton polygon. We set $N_0=\left\lbrace 0\right\rbrace \cup N$.
\begin{lem}\label{key lem 1}
	Suppose $(q^d,v(H_{d}))$ is a vertex of the Newton polygon of $\vp_{\varpi}(X)$. 
The canonical submodule scheme 
is isomorphic to the $A$-module scheme  
	$$
	(Z(\varpi/H_d),u\ \vp^{g} u^{-1}).
	$$
\end{lem}

\begin{proof}
Using the argument in the proof of Proposition~\ref{sub iff subsch}, 
we see that 
$$
 \Phi(X)=X\prod_{\alpha \in N}(X-\alpha)
 $$
belongs to $\cO[X].$
Since $\Phi(X)$ is a distinguished polynomial, Lemma \ref{lemma dist} says that we have an isomorphism of schemes
$\Spec \cO[X]/(\Phi(X)) \cong
\Spec \cO[[X]]/(\Phi(X))$.
Both $\Spec \cO[X]$ and 
$\Spec \cO[[X]]$ are $A$-module schemes via $\varphi$.
Since the coefficients of 
$\Phi(X)$ lie in $\cO$, 
one can check that the $A$-module structure
induces an $A$-module structure on 
$\Spec \cO[X]/(\Phi(X))$
and on 
$\Spec \cO[[X]]/(\Phi(X))$,
using the characterization of $N$
as the set of roots with small valuation.
The isomorphism above induces an isomorphism of $A$-module
schemes 
 $$(\Spec \cO[X]/(\Phi(X)),\vp) \cong (\Spec \cO[[X]]/(\Phi(X)),\vp).
 $$

The $\cO$-algebra homomorphism
$\cO[[X]] \to \cO[[X]]$ 
that sends $X$ to $g(X)$ 
is an isomorphism.    
This induces an isomorphism 
 $$
 (\Spec \cO[[X]]/(\Phi(X)),\vp) 
 \cong 
 (\Spec \cO[[X]]/(X^{q^d}+(\varpi/(1+m){H}_d)X),
(\widehat{\vp} \mid _{A})^{g})
 $$
of $A$-module schemes.

Using Lemma \ref{lemma dist}, we obtain 
an isomorphism of schemes,
and then an isomorphism of $A$-module schemes
 $$
 \left(
\Spec \cO[[X]]
\left/
\left(
X^{q^d}+\frac{\varpi}{(1+m){H}_d}X
\right),(\widehat{\vp} \mid _{A})^{g}
\right)
\right.
 \cong
\left(
Z
\left(
\frac{\varpi}{(1+m){H}_d}
\right), \varphi^g
\right)
 $$
where we set $\varphi^g$ to be 
the $A$-module structure 
induced by the isomorphism.

Finally, the isomorphism 
$\cO[X] \to \cO[X]$
of $\cO$-algebras 
that sends $X$ to $uX$ induces 
$$
\left(
Z
\left(\frac{\varpi}{(1+m){H}_d}
\right)
,\vp^{g}
\right)
 \cong
\left(Z\left(\frac{\varpi}{H_d}\right)
, u\ \vp^{g} u^{-1}
\right).
$$
\end{proof}

For our main theorem, it remains to prove 
the following lemma.
\begin{lem}
\label{key lem 2}
We have an isomorphism of $A$-module schemes
\[
\left(Z\left(\frac{\varpi}{H_d}\right)
, u\ \vp^{g} u^{-1}
\right)
\cong
\left(Z\left(\frac{\varpi}{H_d}\right),
\phi \right).
\]
where $\phi$ is as in Section~\ref{sec:definition of phi}.
\end{lem}
\begin{proof}
Recall that $\phi$ is defined as the composite
$\phi: A \to A_\wp \to A_\wp/\wp A_\wp \to \F \to \End(Z(c))$.
Here, the first is the canonical inclusion, the second is the canonical quotient map,
the third is $\iota^{-1}$ and the last is $\psi$.

The $A$-module structure $u \varphi^g u^{-1}$
factors in a similar manner.

Thus we only need to compare the $\F$-module structures
at the end.
The $\F$-module structure for $\phi$ is given
explicitly in Section~\ref{sec:definition of phi}.
The $\F$-module structure for $\varphi^g$
is given in Remark~\ref{rmk:F-module structure},
and is similar for $u \varphi^g u^{-1}$.
Since the two structures coincide, 
the claim follows.
\end{proof}

\begin{proof}[Proof of Theorem \ref{main thm}]
 The claim follows from Lemmas \ref{key lem 1} and \ref{key lem 2}. 
\end{proof}

\end{document}